\documentclass[12pt]{amsart}
\usepackage{latexsym}
\usepackage{amssymb} 
\usepackage{mathrsfs}
\usepackage{amsmath}
\usepackage{latexsym}
\usepackage{delarray}
\usepackage{amssymb,amsmath,amsfonts,amsthm,mathrsfs}

\usepackage{hyperref}
\renewcommand\eqref[1]{(\ref{#1})} 

 \newtheorem{thm}{Theorem}[section]
 \newtheorem{cor}[thm]{Corollary}
 \newtheorem{lem}[thm]{Lemma}
 \newtheorem{prop}[thm]{Proposition}
 \theoremstyle{definition}
 
 \theoremstyle{remark}
 \newtheorem{rem}[thm]{Remark}
 
 \numberwithin{equation}{section}
\newcommand{\half}{\frac{1}{2}}
\newcommand{\twothird}{\frac{2}{3}}

\newcommand{\ene}{\mathbb{N}}

\newcommand{\ce}{\mathbb{C}}

\newcommand{\zn}{\mathbb{Z}^n}
\newcommand{\tn}{\mathbb{T}^n}

\newcommand{\bi}{\begin{itemize}}
\newcommand{\ei}{\end{itemize}}
\newcommand{\be}{\begin{enumerate}}
\newcommand{\ee}{\end{enumerate}}
\newcommand{\beq}{\begin{equation}}
\newcommand{\eq}{\end{equation}}

\newcommand{\lap}{\mathcal{L}_G}
\newcommand{\cdxi}{\ce^{d_{\xi}\times d_{\xi}}}
\newcommand{\sdxi}{\ce^{d_{\xi}}}

\newcommand{\eigen}{\lambda^2_{[\xi]}}

\def\p#1{{\left({#1}\right)}}

\def\jp#1{{\left\langle{#1}\right\rangle}}

\def\Op{{{\rm Op}}}

\DeclareMathOperator{\Tr}{Tr}

\def\Gh{{\widehat{G}}}

\def\dxi{{d_\xi}}

\def\HS{{\mathtt{HS}}}
\def\Rn{{{\mathbb R}^n}}
\def\Tn{{{\mathbb T}^n}}
\def\Zn{{{\mathbb Z}^n}}
\def\T{{{\mathbb T}^1}}
\def\N{{{\mathbb N}}}
\def\C{{{\mathbb C}}}
\def\SU2{{{\rm SU(2)}}}
\def\SO3{{{\rm SO(3)}}}
\def\lapsu2{{{\mathcal L}_\SU2}}

\def\Op{\text{\rm Op}}

\begin{document}

%
\title[$L^p$-Nuclearity on compact Lie groups]
 {$L^p$-Nuclearity,  traces, and Grothendieck-Lidskii 
 formula on compact Lie groups}

\author[Julio Delgado]{Julio Delgado}

\address{%
Department of Mathematics\\
Imperial College London\\
180 Queen's Gate, London SW7 2AZ\\
United Kingdom
}

\email{j.delgado@imperial.ac.uk}

\thanks{The first author was supported by Marie Curie IIF 301599. The second author was supported by EPSRC Leadership Fellowship EP/G007233/1. }
\author[Michael Ruzhansky]{Michael Ruzhansky}

\address{%
Department of Mathematics\\
Imperial College London\\
180 Queen's Gate, London SW7 2AZ\\
United Kingdom
}

\email{m.ruzhansky@imperial.ac.uk}



\date{\today}
\begin{abstract}
Given a compact Lie group $G$, in this paper we give symbolic criteria for operators to be nuclear and $r$-nuclear on $L^p(G)$-spaces,  with applications to distribution of eigenvalues and trace formulae. Since criteria in terms of kernels are often not effective in view of
Carleman's example, in this paper we adopt the symbolic point of view. The criteria here are given in terms of the concept of matrix symbols defined on the non-commutative analogue of the phase space $G\times\widehat{G}$, where $\widehat{G}$ is the unitary dual of $G$. 
\end{abstract}

\maketitle
\section{Introduction}

Let $G$ be a compact Lie group. In this paper we address the following problems:
\begin{itemize}
\item to find criteria for operators to be nuclear on $L^{p}(G)$,  for $1\leq p<\infty$;
\item since in the Banach spaces, due to Grothendieck's work \cite{gro:me}, 
we know that in order to have the operator trace to agree with the spectral trace,
the notion of nuclearity is not sufficient, to find criteria for the $r$-nuclearity 
($0<r\leq 1$) and 
to apply this to derive information on the spectral behaviour and on the traces of
operators on $L^p(G)$.
\end{itemize}
Our analysis will be based on the global 
quantization recently developed in \cite{rt:book} and \cite{rt:groups} as a noncommutative analogue 
of the Kohn-Nirenberg quantization of operators on $\Rn$. 

In general, for trace class operators in Hilbert spaces, the trace of an operator given by integration of its
integral kernel over the diagonal is equal to the sum of its eigenvalues. However, this property fails
in Banach spaces. The notion of $r$-nuclear operators becomes useful, and Grothendieck
\cite{gro:me} proved that $\frac 23$-nuclear operators in this scale satisfy the Lidskii trace formula
on $L^p$-spaces.
The question of finding good criteria for ensuring the $r$-nuclearity of operators arises but this has
to be formulated in terms different from those on Hilbert  spaces and has to take into account the impossibility of certain kernel
formulations in view of Carleman's example
\cite{car:ex} recalled below. 

In order to get an efficient criterion for the $r$-nuclearity, the application of the notion of a matrix symbol of an operator on a compact Lie will be instrumental. We also give several further applications. A special feature of our criteria is that we do not assume
any regularity condition on the
symbols, which shows a certain advantage in comparison with the traditional
Kohn-Nirenberg quantization in the manifold setting. Here we will completely drop regularity assumption on the symbol as a consequence of the technique of noncommutative quantization that we are using. While Grothendieck's result yields
the same index $2/3$ for all $L^p$ spaces, we relate the index $r$ of the $r$-nuclear operators with 
the index $p$ of $L^p$-spaces in which the trace formula holds. 
Nuclearity criteria for operators on $L^2$ with smooth symbols in H{\"o}rmander classes have been analysed,
see e.g. Shubin \cite[Section 27]{shubin:r}. The problem of finding criteria for
Schatten classes in terms of symbols with lower regularity has been of interest in the last years, see e.g.
\cite{Toft:Schatten-AGAG-2006,Toft:Schatten-modulation-2008,
Buzano-Toft:Schatten-Weyl-JFA-2010}. 

Symbolic criteria for the
$L^{p}$-boundedness of operators on compact Lie groups, the 
Mikhlin-H\"ormander multiplier theorem and its extension to non-invariant
operators for $1<p<\infty$, are presented in
\cite{Ruzhansky-Wirth:Lp-FAA}.

To formulate the notions more precisely, 
let $E$ and $F$ be two Banach spaces and let $0<r\leq 1$. A linear operator $T$
from $E$ to $F$ is called {\em r-nuclear} if there exist sequences
$(x_{n}')\mbox{ in } E' $ and $(y_n) \mbox{ in } F$ so that
\begin{equation}\label{EQ:T-reps}
Tx= \sum_{n=1}^\infty \left<x,x_{n}'\right>y_n 
\end{equation}
and 
\begin{equation}\label{rn}
\sum_{n=1}^\infty \|x_{n}'\|^{r}_{E'}\|y_n\|^{r}_{F} < \infty.
\end{equation}
The class of $r$-nuclear operators is usually endowed with the pseudo-norm 
\[n_r(T)=\inf\{\sum\limits_{n=1}^{\infty}\|x_n'\|_{E'}^r\| y_n\|_{F}^r \}^{\frac{1}{r}},\]
where the infimum is taken over the representations \eqref{EQ:T-reps}
of $T$ such that (\ref{rn}) holds. 
When $r=1$ we obtain the ideal of {\em nuclear operators} and $n_1(\cdot)$ 
is a norm. In this case the definition above agrees with the concept of trace class operators in
the setting of Hilbert spaces ($E=F=H$). 
 
Since we are also interested in the distribution of eigenvalues we shall consider the case $E=F$ and 
the notion of the trace. In order to ensure the existence of a good definition of the trace on the ideal of 
nuclear operators $\mathfrak{N}(E)$ one is led to consider the Banach spaces $E$ enjoying the 
so-called {\em approximation property} (cf. \cite{piet:book}, \cite{df:tensor}). 
It is well known that the spaces $L^p({\Omega},{\mathcal{M}},\mu)$ satisfy the approximation property 
for any measure $\mu$ and $1\leq p\leq\infty $ (cf. \cite[Lemma 19.3.5]{piet:book1}). 
Thus, a Banach space $E$ is said to have the {\em approximation property} if for every compact subset 
$K$ of $E$ and every $\epsilon >0$ there exists a finite rank bounded operator $B$ on $E$ such that
we have
 \[\|x-Bx\|_E<\epsilon \mbox{ for all }x\in K.\]     
 On such spaces, if  
$T:E\rightarrow E$ is nuclear, the trace is well-defined by
$$\Tr(T)=\sum\limits_{n=1}^{\infty}x_{n}'(y_n)  ,$$
where $T=\sum\limits_{n=1}^{\infty}x_{n}'\otimes y_n$ is a
representation of $T$ as in \eqref{EQ:T-reps}. It can be shown that this definition is
independent of the choice of the representation.
 
In the setting of Hilbert spaces the class of $r$-nuclear operators agrees with the Schatten-von 
Neumann ideal of order $r$, a result due to R. Oloff (cf. \cite{Oloff:pnorm}). 
When $r=\frac23$, Grothendieck proved (cf. {\cite{gro:me}) that the trace in Banach spaces
agrees with the sum of all the 
eigenvalues with multiplicities counted. In Hilbert spaces this holds for nuclear (i.e. trace class) operators, 
the result which is known as the Lidskii formula (cf. {\cite{li:formula}). 
It has been proven by A. Pietsch \cite{piet:r} that if $r>1$ the class of operators having decomposition 
(\ref{EQ:T-reps}) and satisfying \eqref{rn} is essentially
reduced to the null operator. The question about the sharpness of the index $r=\twothird$ for trace
formulae in the case of 
$L^p$-spaces has been recently considered by 
Reinov and Laif \cite{Reinov}. Being in the class of $r$-nuclear operators 
can be used to deduce properties concerning the asymptotic behaviour of the corresponding operators. The statement relating Grothendieck's r-nuclearity
result to the Lidskii formula in $L^p$ spaces is known as a 
Grothendieck-Lidskii formula (see e.g. \cite{Reinov}) and we give its
variant on compact Lie groups in Corollary \ref{lem20al} for $0<r\leq 2/3$
and in Corollary \ref{main13ab} for $0<r\leq 1$ with 
$\frac1r=1+\left|\frac12-\frac1p\right|$.
The $r$-nuclear operators are sometimes known as $p$-nuclear operators, but here we will reserve the index $p$ 
to indicate the $L^p$-spaces. A description of the current state of the art of the general theory of $p$-nuclear 
operators has recently appeared in Hinrichs and Pietsch \cite{hi:pn}.  

Among other things, in this paper
we establish sufficient conditions on the 
matrix-valued symbol of an operator in order to ensure the $r$-nuclearity in $L^p$-spaces. 
The nuclearity of pseudo-differential operators on the circle $\T$ has been
recently analysed in \cite{dw:trace} but the situation in the present paper is much more
subtle because of the necessarily appearing multiplicities of the eigenvalues of the Laplacian on
the noncommutative compact Lie groups; moreover, due to the commutativity of the torus, the
symbol there is scalar and hence all of its ``matrix''-norms are uniformly equivalent which is not
the case if the group $G$ is noncommutative.

We shall now briefly recall a classical result of 
Carleman (\cite{car:ex}) which will be helpful to clarify the significance of our symbolic criteria. In 1916 Torsten Carleman 
 constructed a periodic continuous function 
$\varkappa(x)=\sum\limits_{n=-\infty}^{\infty}c_ne^{2\pi inx}$, i.e. a continuous function on the 
commutative Lie group $\T$, for which the Fourier coefficients $c_n$ satisfy
$$\sum\limits_{n=-\infty}^{\infty} |c_n|^r=\infty\qquad \textrm{ for any } r<2.$$     
Now, considering the normal operator 
\begin{equation}\label{EQ:Carleman}
Tf=f*\varkappa
\end{equation}
acting on $L^2(\T)$ 
one obtains that the sequence $(c_n)_n$ forms a complete system of eigenvalues of 
this operator corresponding to the complete orthonormal system $\phi_n(x)=e^{2\pi inx}$, 
$T\phi_n=c_n\phi_n$. The system $\phi_n$ is also complete for $T^*$, $T^*\phi_n=\overline{c_n}\phi_n$, 
the singular values of $T$ are given by $ s_n(T)=|c_n|$ and
hence
$$\sum\limits_{n=-\infty}^{\infty}s_n(T)^r=\infty ,$$ 
 for $r<2$. Hence, the operator $T$ is not nuclear. Moreover, due to the aforementioned 
 Oloff's result the operator $T$ is not $r$-nuclear for $0<r\leq 1$.  
 However, the continuous integral
 kernel $k(x,y)=\varkappa (x-y)$ satisfies any kind of integral condition of the form 
 $ \int\int |k(x,y)|^s dxdy<\infty$ due to the boundedness of $k$.  This shows that it is impossible to formulate a sufficient condition of this type for the kernel ensuring nuclearity on the torus $\T$. 
 
 In this work we will establish conditions imposed on symbols 
 instead of kernels ensuring the $r$-nuclearity of the corresponding operators. 
 The criteria that we will obtain in the general case for the nuclearity from
 $L^{p_1}(G)$ to $L^{p_2}(G)$ will depend on whether $p_1\leq 2$ or $p_1\geq 2$.
 For $p_1=p_2$ we will apply this to the question of the convergence of the
 series of eigenvalues of operators and to the validity of the Lidskii formula.
 We give examples applying our results to the heat kernel on general compact Lie
 groups (Subsection \ref{SEC:heat})
 as well as to the Laplacian and the sub-Laplacian on $\SU2\simeq\mathbb S^{3}$ 
 and on $\SO3$ (Subsection \ref{SEC:SU2}).

In Section \ref{SEC:Prelim} we discuss and formulate the known criteria for nuclearity as
well as make a short introduction to the global quantization on compact Lie groups.
In Section \ref{SEC:nuclearity-p} we move to the setting of $L^p$-spaces and
formulate our criteria for the $r$-nuclearity. There are different possibilities
of how to impose conditions on the symbol. We will discuss both the cases of invariant and
non-invariant operators, and give examples of our results on the tori, on the group
$\SU2$ and on $\SO3$. In Section \ref{SEC:nuclearity-pr} we give  applications to summability of eigenvalues, trace formulae and the Lidskii theorem. 
In particular, in Subsection \ref{SEC:heat}
we give the example of the trace kernel and its trace. 

The authors would like to thank Jens Wirth for discussions and remarks.

\section{Preliminaries}  
\label{SEC:Prelim}

In this section we recall some basic facts about the concepts of nuclear and $r$-nuclear operators, 
and the notion of the trace on Banach spaces. In particular, we consider the trace of 
nuclear operators on $L^p(\mu)$. The fact that these spaces satisfy the approximation property 
is a classical result (cf. \cite{gro:me}, \cite{piet:book1}). We refer the reader to \cite{piet:book1} 
and to \cite[Chapter 4.2]{piet:book} for the general theory of traces on operator ideals and the notation 
used in this section, see also \cite{goh:trace} for an exposition on the distribution of the eigenvalues. 
For the theory of pseudo-differential operators on compact Lie groups the we refer to 
\cite{rt:book} and \cite{rt:groups}.  
 
In the case of $L^p$-spaces we first record the following characterisation of nuclear operators 
(cf. \cite{del:tracetop}). In the statement below we shall consider $({\Omega}_1,{\mathcal{M}}_1,\mu_1 )$ and
$({\Omega}_2,{\mathcal{M}}_2,{\mu}_2)$ to be two $\sigma$-finite measure spaces. 

\begin{thm}\label{ch2} 
Let $1\leq p_1,p_2 <\infty$ and let $q_1$ be such that
$\frac{1}{p_1}+\frac{1}{q_1}=1$. 
 An operator $T:L^{p_1}({\mu}_1)\rightarrow L^{p_2}({\mu}_2)$ is nuclear if and only if 
 there exist sequences
 $(g_n)_n$ in $L^{p_2}({\mu}_2)$, and $(h_n)_n$ in $L^{q_1}(\mu_1)$ such that
  $\sum \limits_{n=1}^\infty \| g_n\|_{L^{p_2}}
 \|h_n\|_{L^{q_1}}<\infty$, and such that for all $f\in L^{p_1}(\mu_1)$ we have
$$Tf(x)=\int\left(\sum\limits_{n=1}^{\infty}
  g_n(x)h_n(y)\right)f(y)d\mu_1(y), \quad \mbox{for  a.e }x.$$
\end{thm}

\begin{rem}\label{remrn}
 An analogue of the characterisation above holds for $r$-nuclear operators replacing the terms $\| g_n\|_{L^{p_2}}
 \|h_n\|_{L^{q_1}}$ by $\| g_n\|_{L^{p_2}}^r
 \|h_n\|_{L^{q_1}}^r$ in the sum. 
\end{rem}

A distribution of the eigenvalues for $r$-nuclear operators can be obtained from  the next theorem 
relating the eigenvalues and the class of $r$-nuclear operators 
(cf. \cite[Chap. II, p. 16]{gro:me}, \cite[Chap. 5, Theorem 4.2]{goh:trace}):

\begin{thm}\label{THM:evs}
Let $E$ be a Banach space which has the approximation property. Let $T$ be an $r$-nuclear operator from $E$ into $E$
 for some $0<r\leq 1$. Then
\[\sum\limits_{n=1}^{\infty}|\lambda_n(T)|^s\leq n_r^s(T),\qquad \frac{1}{s}=\frac{1}{r}-\half,\]
where $\lambda_n(T)$ denote the eigenvalues of $T$ with multiplicities counted.
\end{thm}

\begin{rem} (i) Note that from $\frac{1}{s}=\frac{1}{r}-\half$ we obtain that $s=\frac{2r}{2-r}$ for $0<r\leq 1$. In 
particular, the function  
$s(r)=\frac{2r}{2-r}$ has the range $(0,2]$. It is clear that if $s>2$ the series on the left in Theorem \ref{THM:evs} 
also converges
 but the interesting situation is to find smaller values of $s$ ensuring such convergence.

(ii) Theorem \ref{THM:evs} was established by Grothendieck \cite{gro:me}, and later extended by e.g.
K\"onig (\cite[page 107]{ko:pn}) to the scale of Lorentz sequences spaces; see also 
\cite[Theorem 3.8.6]{piet:book}.

(iii) If $r=1$ we get $s=2$, a classical result by Grothendieck (cf. \cite{gro:me}) establishing the square 
summability of eigenvalues for nuclear operators. It is also known by Grothendieck that $s=2$ 
is the best possible exponent in this case.  
\end{rem}

Theorem \ref{THM:evs} will be applied jointly with our sufficient conditions for $r$-nuclearity, 
to obtain estimates on the asymptotic behaviour of the eigenvalues. 
From this point of view, the main goal of this paper becomes to find suitable criteria for
ensuring the $r$-nuclearity of an operator.

\medskip
Given a compact Lie group $G$, in this work we consider ${\Omega}_1={\Omega}_2=G$  and 
${\mathcal{M}}={\mathcal{M}}_1={\mathcal{M}}_2$, the Borel $\sigma$-algebra associated to the topology of 
the smooth manifold $G$, with $\mu=\mu_1=\mu_2$ the normalised Haar measure of $G$. 
Let $\widehat{G}$ denote the set of equivalence classes of continuous irreducible unitary 
representations of $G$. Since $G$ is compact, the set $\widehat{G}$ is discrete.  
For $[\xi]\in \widehat{G}$, by choosing a basis in the representation space of $\xi$, we can view 
$\xi$ as a matrix-valued function $\xi:G\rightarrow \ce^{d_{\xi}\times d_{\xi}}$, where 
$d_{\xi}$ is the dimension of the representation space of $\xi$. 
By the Peter-Weyl theorem the collection
$$
\left\{ \sqrt{d_\xi}\,\xi_{ij}: \; 1\leq i,j\leq d_\xi,\; [\xi]\in\Gh \right\}
$$
is an orthonormal basis of $L^2(G)$.
If $f\in L^1(G)$ we define its global Fourier transform at $\xi$ by 
\begin{equation}\label{EQ:FG}
\mathcal F_G f(\xi)\equiv \widehat{f}(\xi):=\int_{G}f(x)\xi(x)^*dx,
\end{equation}
where $dx$ is the normalised Haar measure on $G$. Thus, if $\xi$ is a matrix representation, 
we have $\widehat{f}(\xi)\in\ce^{d_{\xi}\times d_{\xi}} $. The Fourier inversion formula is a consequence
 of the Peter-Weyl theorem, so that we have
\beq \label{EQ:FGsum}
f(x)=\sum\limits_{[\xi]\in \widehat{G}}d_{\xi} \Tr(\xi(x)\widehat{f}(\xi)).
\eq
Given a sequence of matrices $a(\xi)\in\mathbb C^{d_\xi\times d_\xi}$, we can define
\begin{equation}\label{EQ:FGi}
(\mathcal F_G^{-1} a)(x):=\sum\limits_{[\xi]\in \widehat{G}}d_{\xi} \Tr(\xi(x) a(\xi)),
\end{equation}
 where the series can be interpreted distributionally or absolutely depending on the growth of 
(the Hilbert-Schmidt norms of) $a(\xi)$. For a further discussion we refer the reader to \cite{rt:book}.

For each $[\xi]\in \widehat{G}$, the matrix elements of $\xi$ are the eigenfunctions for the Laplacian $\mathcal{L}_G$ 
(or the Casimir element of the universal enveloping algebra), with the same eigenvalue which we denote by 
$-\lambda^2_{[\xi]}$, so that we have
\begin{equation}\label{EQ:Lap-lambda}
-\mathcal{L}_G\xi_{ij}(x)=\lambda^2_{[\xi]}\xi_{ij}(x)\qquad\textrm{ for all } 1\leq i,j\leq d_{\xi}.
\end{equation} 
The weight for measuring the decay or growth of Fourier coefficients in this setting is 
$\jp{\xi}:=(1+\lambda^2_{[\xi]})^{\half}$, the eigenvalues of the elliptic first-order pseudo-differential operator 
$(I-\mathcal{L}_G)^{\half}$.
The Parseval identity takes the form 
\begin{equation}\label{EQ:Parseval}
\|f\|_{L^2(G)}= \left(\sum\limits_{[\xi]\in \widehat{G}}d_{\xi}\|\widehat{f}(\xi)\|^2_{\HS}\right)^{1/2},\quad
\textrm{ where }
\|\widehat{f}(\xi)\|^2_{\HS}=\Tr(\widehat{f}(\xi)\widehat{f}(\xi)^*),
\end{equation}
which gives the norm on 
$\ell^2(\widehat{G})$. 

For a linear continuous operator $A$ from $C^{\infty}(G)$ to $\mathcal{D}'(G) $ 
we define  its {\em matrix-valued symbol} $\sigma_A(x,\xi)\in\cdxi$ by 
\begin{equation}\label{EQ:A-symbol}
\sigma_A(x,\xi):=\xi(x)^*(A\xi)(x)\in\cdxi,
\end{equation}
where $A\xi(x)\in \cdxi$ is understood as $(A\xi(x))_{ij}=(A\xi_{ij})(x)$, i.e. by 
applying $A$ to each component of the matrix $\xi(x)$.
Then one has (\cite{rt:book}, \cite{rt:groups}) the global quantization
\begin{equation}\label{EQ:A-quant}
Af(x)=\sum\limits_{[\xi]\in \widehat{G}}d_{\xi}\Tr(\xi(x)\sigma_A(x,\xi)\widehat{f}(\xi))
\end{equation}
in the sense of distributions, and the sum is independent of the choice of a representation $\xi$ from each 
equivalence class 
$[\xi]\in \widehat{G}$. If $A$ is a linear continuous operator from $C^{\infty}(G)$ to $C^{\infty}(G)$,
the series \eqref{EQ:A-quant} is absolutely convergent and can be interpreted in the pointwise
sense. We will also write $A=\Op(\sigma_A)$ for the operator $A$ given by
the formula \eqref{EQ:A-quant}. The symbol $\sigma_A$ can be interpreted as a matrix-valued
function on $G\times\widehat{G}$.
We refer to \cite{rt:book, rt:groups} for the consistent development of this quantization
and the corresponding symbolic calculus. If the operator $A$ is left-invariant then its symbol
$\sigma_A$ does not depend on $x$. We often call such operators simply invariant.

We now record simple inequalities on the norms of the representation coefficients which will be essential for
 the analysis of $r-$ nuclearity:

\begin{lem} \label{LEM:xiests}
Let $G$ be a compact Lie group and let $[\xi]\in\Gh$. Then for all
$1\leq i,j\leq d_\xi$ we have 
\begin{equation}\label{ineqA}
 \|\xi_{ij}\|_{L^{q}(G)} \leq\left\{
\begin{array}{rl}
d_{\xi}^{-\frac{1}{q}} ,& 2\leq q\leq \infty,\\
d_{\xi}^{-\frac{1}{2}} , & 1\leq q\leq 2,
\end{array} \right.
\end{equation}
where for $q=\infty$ we adopt the usual convention $d_{\xi}^{-\frac{1}{q}}=1$.
\end{lem}
\begin{proof} 
If $q=\infty$, for any $y\in G$ we have $|\xi(y)_{ij}|\leq \|\xi(y)\|_{op}=1$ by the unitarity of representations in $\Gh$.
If $2\leq q<\infty$ we apply the inequality 
$$\|f\|_{L^{q}}\leq \|f\|_{L^{\infty}}^{\frac{q-2}{q}}\|f\|_{L^{2}}^{\frac{2}{q}}.$$
Using that $\sqrt{d_\xi}\ \xi_{ij}$ is an orthonormal set in $L^2(G)$, i.e. that
$\|\xi_{ij}\|_{L^2}=d_\xi^{-\half}$, and that we have just showed that $\|\xi_{ij}\|_{L^\infty}\leq 1$,
we obtain
$$
\|\xi_{ij}\|_{L^{q}(G)} \leq  \|\xi_{ij}\|_{L^{\infty}}^{\frac{q-2}{q}}\|\xi_{ij}\|_{L^{2}}^{\frac{2}{q}}
\leq \|\xi_{ij}\|_{L^{2}}^{\frac{2}{q}}
\leq d_{\xi}^{-\frac{1}{q}}.
$$
Finally, for $1\leq q\leq 2$, using H\"older's inequality, we get 
 \begin{align*}
\|\xi_{ij}\|_{L^{q}(G)}^{q} = & \int_{G}|\xi_{ij}(y)|^{q}dy\\
 \leq &  \left(\int_{G}1dy\right)^{1-\frac{q}{2}} \left(\int_{G}|\xi_{ij}(y)|^{q\frac{2}{q}}dy\right)^{\frac{q}{2}}\\
\leq & \|\xi_{ij}\|_{L^{2}(G)}^{{q}}
= d_{\xi}^{-\frac{q}{2}},
\end{align*}
where we have used that the Haar measure on $G$ is normalised.
\end{proof}

Our criteria will be formulated in terms of norms of the matrix-valued symbols. In order to justify the appearence of them, we recall that if $A\in\Psi^m(G)$ is a pseudo-differential operators in H\"ormander's class
$\Psi^m(G)$, i.e. if all of its localisations to $\Rn$ are pseudo-differential operators with
symbols in the class $S^m_{1,0}(\Rn)$, then the matrix-symbol of $A$ satisfies
$$\|\sigma_A(x,\xi)\|_{op}\leq C \jp{\xi}^m \qquad \textrm {for all } x\in G,\; [\xi]\in\Gh.$$
Here $\|\cdot\|_{op}$ denotes the operator norm of the matrix
multiplication by the matrix $\sigma_A(x,\xi).$
For this fact, see e.g. \cite[Lemma 10.9.1]{rt:book} or \cite{rt:groups}, and for the complete
characterisation of H\"ormander classes $\Psi^m(G)$ in terms of matrix-valued symbols
see also \cite{Ruzhansky-Turunen-Wirth:arxiv}. In particular, this motivates the
 employ of the operator norms of the matrix-valued symbols. However, since $\sigma_{A}$ 
is in general a matrix, other matrix norms become useful as well.

\section{$r$-Nuclearity on $L^p(G)$ and examples} 
\label{SEC:nuclearity-p}
In this and next sections we analyse the $r$-nuclearity and trace formulae in $L^p$-spaces.
 We recall that the case $r=1$ correspond to the class of nuclear operators. One of the features of the obtained criteria is that they require the integrability (in some $L^p$-space) of symbols
$\sigma_A(x,\xi)$ with respect to $x$ but do not assume any regularity of the symbol.

We start by proving the following sufficient condition for $r$-nuclearity of operators on $L^2(G)$ with symbols depending only on $\xi$. We note that the property
that the symbol depends only on $\xi$ means that the operator is left-invariant, that is, it commutes with the left translations on the group $G$.
\begin{thm}\label{main125} 
Let $G$ be a compact Lie group. 
 Let $A:L^{2}(G)\to L^{2}(G)$ be a linear
continuous operator with matrix-valued symbol $\sigma_A(\xi)$
depending only on $\xi.$ Then
 $A$ is $r$-nuclear provided that its symbol $\sigma_A$ satisfies 
\beq\sum_{[\xi]\in\Gh}d_{\xi} \|\sigma_A(\xi)\|_{S_r}^r<\infty.\label{chmult}\eq

Here $\|\sigma_A(\xi)\|_{S_r}=\left(\Tr (|\sigma_{A}(\xi)|^r)\right)^{1/r}$ is the Schatten-norm of order 
$r$ of the matrix $\sigma_A(\xi)\in\cdxi$.
\end{thm}
\begin{proof} Let us suppose that the symbol $\sigma_A$ satisfies (\ref{chmult}).  
We note that the kernel of the operator $A$ is given by 
\[k(x,y)=\sum_{[\xi]\in\widehat{G} }d_{\xi} \Tr(\xi(x)\sigma_A(\xi)\xi(y)^* ),\]
and we will show that it is well-defined and has the tensor product form of Theorem \ref{ch2} that 
is required for the nuclearity. To abbreviate the notation, we will write
$\sigma(\xi)$ for $\sigma_A(\xi)$.
We begin by writing
\[ \Tr(\xi(x)\sigma(\xi)\xi(y)^* )=\sum\limits_{i,j=1}^{d_{\xi}}(\xi(x)\sigma(\xi))_{ij}\overline{\xi(y)}_{ij},\]
and we set 
\[g_{\xi,ij}(x)=d_{\xi}(\xi(x)\sigma(\xi))_{ij},\,h_{\xi,ij}(y)=(\xi(y)^*)_{ji}=\overline{\xi(y)}_{ij}.\]
For $g_{\xi,ij}(x)$ we have  
\begin{align}
\|g_{\xi,ij}\|_{L^{2}(G)}=&\sqrt{d_{\xi}}\|\sqrt{d_{\xi}}(\xi(x)\sigma(\xi))_{ij}\|_{L^{2}(G)}\nonumber\\
= & \sqrt{d_{\xi}}\|\sum\limits_{k=1}^{d_{\xi}}\sqrt{d_{\xi}}(\xi(x))_{ik}\sigma(\xi)_{kj}\|_{L^{2}(G)}\nonumber\\
= & \sqrt{d_{\xi}}\left(\sum\limits_{k=1}^{d_{\xi}}\sqrt{d_{\xi}}(\xi(x))_{ik}\sigma(\xi)_{kj}, \sum\limits_{k'=1}^{d_{\xi}}\sqrt{d_{\xi}}(\xi(x))_{ik'}\sigma(\xi)_{k'j}\right)_{L^{2}(G)}^{\half}\nonumber\\
= &\sqrt{d_{\xi}}\left(\sum\limits_{k=1}^{d_{\xi}}\sigma(\xi)_{kj}\overline{\sigma(\xi)_{kj}}\right)^{\half}\nonumber\\
= &\sqrt{d_{\xi}}\left(\sum\limits_{k=1}^{d_{\xi}}(\sigma(\xi)^*)_{jk}\sigma(\xi)_{kj}\right)^{\half}\nonumber\\
= &\sqrt{d_{\xi}}\left(\sigma(\xi)^*\sigma(\xi)\right)^\half _{jj}\nonumber\\
= &\sqrt{d_{\xi}}|\sigma(\xi)|_{jj}.\label{diagnorm}
\end{align}
Hence $\|g_{\xi,ij}\|_{L^{2}(G)}^r={d_{\xi}}^{\frac r2}|\sigma(\xi)|_{jj}^r.$

Now, since $\{d_{\xi}^\half{\xi}_{ij}\}$ is an orthonormal set in $L^2(G)$, we have
\[\|\overline{\xi}_{ij}\|_{L^{2}(G)}^r= d_{\xi}^{-\frac{r}{2}}.\]
Therefore, 
\begin{align*}
\sum\limits_{\xi, ij}\|g_{\xi,ij}(\cdot)\|_{L^{2}(G)}^r\|h_{\xi,ij}(\cdot)\|_{L^{2}(G)}^r &\leq \sum\limits_{\xi}d_{\xi}^{\frac r2} d_{\xi}^{-\frac{r}{2}} \sum\limits_{i,j=1}^{d_{\xi}}|\sigma(\xi)|_{jj}^r\\
&= \sum\limits_{\xi}\sum\limits_{j=1}^{d_{\xi}}d_{\xi}|\sigma(\xi)|_{jj}^r\\
&= \sum\limits_{\xi}d_{\xi}\Tr(|\sigma(\xi)|^r)\\
&= \sum\limits_{\xi}d_{\xi}\|\sigma(\xi)\|_{S_r}^r\\
&< \,\infty,
\end{align*}
completing the proof.
\end{proof}
\begin{rem} We point out that one can prove  that the condition (\ref{chmult}) ensuring the $r$-nuclearity for left-invariant operators on $L^2(G)$ is also 
 necessary. We recall that in the Hilbert space setting, the Schatten class of order $r$ agrees with the class of $r$-nuclear operators whenever $0<r\leq 1$ by Oloff's result \cite{Oloff:pnorm}. For the details of Schatten classes of invariant operators on compact Lie groups we refer the reader to the recent work \cite{dr13:schatten}.
\end{rem} 
We will extend Theorem \ref{main125} to the setting of $L^p(G)$ spaces. We shall require the following notation:    
 $\ell^{\infty}$ denotes the $L^{\infty}$-norm  on $\sdxi$ and 
$\|\cdot\|_{op(\ell^{\infty},\ell^{\infty})}$ denotes the operator norm with respect to $\ell^{\infty}$ on $\sdxi$. 
More precisely, for each $d\in\N$, let $B\in\C^{d\times d}$ and $u\in\C^d$. Denoting
$$\|B\|_{op(\ell^{\infty},\ell^{\infty})}:=\max_{1\leq i\leq d} \sum_{j=1}^d |B_{ij}|,$$
we have
$|(Bu)_i|\leq \sum_{j=1}^d |B_{ij}| \max_{1\leq j\leq d} |u_j|\leq \|B\|_{op(\ell^{\infty},\ell^{\infty})}\|u\|_{\ell^\infty},$
so that we get
$$\|Bu\|_{\ell^\infty}\leq \|B\|_{op(\ell^{\infty},\ell^{\infty})}\|u\|_{\ell^\infty},$$
justifying the notation $\|\cdot\|_{op(\ell^{\infty},\ell^{\infty})}$, and the appearance of this norm.
The transpose of the matrix $M$ will be denoted by $M^t.$ 
We first deal with left-invariant operators.

\begin{thm}\label{main126a} 
Let $G$ be a compact Lie group and let $1\leq p_1,p_2 <\infty$, $0<r\leq 1$.
Let $A:L^{p_1}(G)\to L^{p_2}(G)$ be a linear
continuous operator with matrix-valued symbol $\sigma_A(\xi)$ depending only on $\xi$. Then, 
if $1\leq p_2\leq 2$ and  
$$\sum_{[\xi]\in\Gh}d_{\xi}^{1+(\frac{1}{\widetilde{p_1}}-\frac{1}{2})r}\|\sigma_A(\xi)\|_{S_r}^r<\infty,$$
where $\widetilde{p_1}=\min\{2,p_1\}$, then the operator $A:L^{p_1}(G)\rightarrow L^{p_2}(G)$ is $r$-nuclear. 


If  $p_2>2$ and
\[\sum_{[\xi]\in\Gh}d_{\xi}^{1+(\frac{1}{\widetilde{p_1}}-\frac{1}{p_2})r} 
\|(\sigma_A(\xi))^t\|_{op(\ell^{\infty},\ell^{\infty})}^{\frac{p_2-2}{p_2}r}
\|\sigma_A(\xi)\|_{S_{\frac{2r}{p_2}}}^{\frac{2r}{p_2}} <\infty,\]
then $A:L^{p_1}(G)\rightarrow L^{p_2}(G)$ is $r$-nuclear.
\end{thm}
\begin{proof} Let $q_1$ be such that $\frac{1}{p_1}+\frac{1}{q_1}=1$. We will consider $\widetilde{q_1}=\max\{2,q_1\}$ and we observe that
 $\frac{1}{\widetilde{q_1}}+\frac{1}{\widetilde{p_1}}=1$. If $1\leq p_2 \leq 2$ we have, using
 \eqref{diagnorm}, and denoting $\sigma=\sigma_A$,
\[\|(\xi(x)\sigma(\xi))_{ij}\|_{L^{p_2}(G)}^r\leq \|(\xi(x)\sigma(\xi))_{ij}\|_{L^2(G)}^r\leq d_{\xi}^{-\frac{r}{2}}|\sigma(\xi)|_{jj}^r.\]
On the other hand
\[\|\overline{\xi}_{ij}\|_{L^{q_1}(G)}^r\leq d_{\xi}^{-\frac{r}{\widetilde{q_1}}}.\]
Therefore,
\begin{align*}
\sum\limits_{\xi, ij}d_{\xi}^r\|(\xi(x)\sigma(\xi))_{ij}\|_{L^{p_2}(G)}^r\|\overline{\xi(y)}_{ij}\|_{L^{q_1}(G)}^r\leq &\sum\limits_{\xi}d_{\xi}^rd_{\xi}^{-\frac r2}d_{\xi}^{-\frac{r}{\widetilde{q_1}}}\sum\limits_{ij}|\sigma(\xi)|_{jj}^r\\
= & \sum\limits_{\xi}d_{\xi}^{1+(\frac{1}{2}-\frac{1}{\widetilde{q_1}})r}\|\sigma_A(\xi)\|_{S_r}^r\\
= & \sum\limits_{\xi}d_{\xi}^{1+(\frac{1}{\widetilde{p_1}}-\frac{1}{2})r}\|\sigma_A(\xi)\|_{S_r}^r\\
<&\infty.
\end{align*}

We now suppose $p_2>2$. 
Now, if $p_2>2$ we first observe that
\begin{align*}
(\xi(x)\sigma(\xi))_{ij}=\sum\limits_{k=1}^{\dxi}\xi(x)_{ik}\sigma(\xi)_{kj}
=\sum\limits_{k=1}^{\dxi}(\sigma(\xi))_{jk}^t\xi(x)_{ik}.
\end{align*}
Hence and taking into account that $|\xi(x)_{ik}|\leq 1 $, we get
\begin{align}
|(\xi(x)\sigma(\xi))_{ij}|=&|\sum\limits_{k=1}^{\dxi}(\sigma(\xi))_{jk}^t\xi(x)_{ik}|\nonumber\\
\leq&\|(\sigma(\xi))^t\|_{op(\ell^{\infty},\ell^{\infty})}\|(\xi(x)_{i1},\cdots,\xi(x)_{i\dxi})\|_{\ell^{\infty}}\nonumber\\
\leq&\|(\sigma(\xi))^t\|_{op(\ell^{\infty},\ell^{\infty})}.\label{toper}
\end{align}

Then using \eqref{diagnorm} and \eqref{toper} we obtain 
\begin{align*}
\|(\xi(x)\sigma(\xi))_{ij}\|_{L^{p_2}(G)}^r=&\left(\int_{G}|(\xi(x)\sigma(\xi))_{ij}|^{p_2} dx\right)^{\frac{r}{p_2}}\\
=&\left(\int_{G}|(\xi(x)\sigma(\xi))_{ij}|^{p_2-2}|(\xi(x)\sigma(\xi))_{ij}|^{2} dx\right)^{\frac{r}{p_2}}\\
\leq & \sup\limits_{x}|(\xi(x)\sigma(\xi))_{ij}|^{\frac{p_2-2}{p_2}r}\|(\xi(x)\sigma(\xi))_{ij}\|_{L^2(G)}^{\frac{2r}{p_2}}\\
\leq & \|(\sigma(\xi))^t\|_{op(\ell^{\infty},\ell^{\infty})}^{\frac{p_2-2}{p_2}r}d_{\xi}^{-\frac{r}{p_2}}|\sigma(\xi)|_{jj}^{\frac{2r}{p_2}}.
\end{align*}
Therefore,
\begin{align*}
& \sum\limits_{\xi, ij}d_{\xi}^r\|(\xi(x)\sigma(\xi))_{ij}\|_{L^{p_2}(G)}^r\|\overline{\xi(y)}_{ij}\|_{L^{q_1}(G)}^r
\\ \leq &\sum\limits_{\xi}d_{\xi}^rd_{\xi}^{-\frac{r}{p_2}}d_{\xi}^{-\frac{r}{\widetilde{q_1}}} \|(\sigma(\xi))^t\|_{op(\ell^{\infty},\ell^{\infty})}^{\frac{p_2-2}{p_2}r}\sum\limits_{ij}|\sigma(\xi)|_{jj}^{\frac{2r}{p_2}}\\
= &\sum\limits_{\xi}d_{\xi}^rd_{\xi}^{-\frac{r}{p_2}}d_{\xi}^{-\frac{r}{\widetilde{q_1}}} \|(\sigma(\xi))^t\|_{op(\ell^{\infty},\ell^{\infty})}^{\frac{p_2-2}{p_2}r}d_{\xi}\sum\limits_{j}|\sigma(\xi)|_{jj}^{\frac{2r}{p_2}}\\
= &\sum\limits_{\xi}d_{\xi}d_{\xi}^{(1-\frac{1}{\widetilde{q_1}}-\frac{1}{p_2})r}\|(\sigma(\xi))^t\|_{op(\ell^{\infty},\ell^{\infty})}^{\frac{p_2-2}{p_2}r}\Tr(|\sigma(\xi)|^{\frac{2r}{p_2}})\\
= &\sum\limits_{\xi}d_{\xi}d_{\xi}^{(\frac{1}{\widetilde{p_1}}-\frac{1}{p_2})r} \|(\sigma(\xi))^t\|_{op(\ell^{\infty},\ell^{\infty})}^{\frac{p_2-2}{p_2}r}\|\sigma(\xi)\|_{S_{\frac{2r}{p_2}}}^{\frac{2r}{p_2}}\\
<&\infty,
\end{align*}
completing the proof.
\end{proof}
In the particular case of diagonal symbols only depending on $\xi$ we can improve the sufficient condition 
in the above theorem. An example of such behaviour is the sub-Laplacian on $G$ 
that always has a diagonal
symbol in an appropriately chosen basis in the representation spaces.
Moreover, symbols of
left-invariant self-adjoint operators can be chosen to be diagonal by choosing a 
particular representative from each equivalence class $[\xi]\in\Gh$.
We formulate a general result now, and will give its application to the sub-Laplacian
in Subsection \ref{SEC:SU2}.

\begin{thm}\label{main126} 
Let $G$ be a compact Lie group, $1\leq p_1,p_2 <\infty$ and $0<r\leq 1$. Let $A:L^{p_1}(G)\to L^{p_2}(G)$ be a linear
continuous operator with matrix-valued diagonal symbol $\sigma_A(\xi)$ depending only on $\xi$.   
 Assume that   
$$\sum_{[\xi]\in\Gh}d_{\xi}^{1+(\frac{1}{\widetilde{p_1}}-\frac{1}{\widetilde{p_2}})r} \|\sigma_A(\xi)\|_{S_r}^r<\infty,$$
where $\widetilde{p_1}=\min\{2,p_1\}, \widetilde{p_2}=\max\{2,p_2\}$. Then the operator $A:L^{p_1}(G)\rightarrow L^{p_2}(G)$ is $r$-nuclear. 
\end{thm}
\begin{proof}  We consider $q_1$ such that $\frac{1}{p_1}+\frac{1}{q_1}=1$ and $\widetilde{q_1}=\max\{2,q_1\}$, and we observe that
 $\frac{1}{\widetilde{q_1}}+\frac{1}{\widetilde{p_1}}=1$. Since $\sigma(\xi)=\sigma_A(\xi)$ is diagonal, and 
 using \eqref{diagnorm}, we have 
\begin{align*}
\|(\xi(x)\sigma(\xi))_{ij}\|_{L^{p_2}(G)}^r=\|\xi(x)_{ij}\sigma(\xi)_{jj}\|_{L^{p_2}(G)}^r
\leq d_{\xi}^{-\frac{r}{\widetilde{p_2}}}|\sigma(\xi)|_{jj}^r.
\end{align*}
On the other hand, by (\ref{ineqA}) we have 
\[\|\overline{\xi}_{ij}\|_{L^{q_1}(G)}^r\leq d_{\xi}^{-\frac{r}{\widetilde{q_1}}}.\]
Therefore,
\begin{align*}
\sum\limits_{\xi, ij}d_{\xi}^r\|(\xi(x)\sigma(\xi))_{ij}\|_{L^{p_2}(G)}^r\|\overline{\xi(y)}_{ij}\|_{L^{q_1}(G)}^r\leq &\sum\limits_{\xi}d_{\xi}^rd_{\xi}^{-\frac{r}{\widetilde{p_2}}}d_{\xi}^{-\frac{r}{\widetilde{q_1}}}\sum\limits_{ij}|\sigma(\xi)|_{jj}^r\\
= & \sum\limits_{\xi}d_{\xi}^{-\frac{r}{\widetilde{p_2}}}d_{\xi}^{\frac{r}{\widetilde{p_1}}}\sum\limits_{j}d_{\xi}|\sigma(\xi)|_{jj}^r\\
= & \sum\limits_{\xi}d_{\xi}^{1+(\frac{1}{\widetilde{p_1}}-\frac{1}{\widetilde{p_2}})r}\|\sigma(\xi)\|_{S_r}^r\\
<&\infty,
\end{align*}
completing the proof.
\end{proof}

We will sometimes give examples of our results on the torus, so we summarise its notation:
\begin{rem}\label{REM:torus}
If $G=\Tn=\Rn/ \Zn$, we have $\widehat{\Tn}\simeq \Zn$, and the collection $\{ \xi_k(x)=e^{2\pi i x\cdot k}\}_{k\in\Zn}$ 
is the orthonormal basis of $L^2(\Tn)$, and all $d_{\xi_k}=1$. 
If an operator $A$ is invariant on $\Tn$, its symbol becomes
$\sigma_A(\xi_k)=\xi_k(x)^* A\xi_k(x)=A\xi_k(0).$ 
In general, on the torus we will often simplify the notation by identifying $\widehat{\Tn}$ with $\Zn$,
and thus writing $\xi\in\Zn$ instead of $\xi_k\in\Zn$. The toroidal quantization 
\begin{equation}\label{EQ:A-torus}
Af(x)=\sum\limits_{\xi\in\Zn} e^{2\pi i x\cdot \xi} \sigma_A(x,\xi)\widehat{f}(\xi)
\end{equation}
has been analysed extensively in \cite{Ruzhansky-Turunen-JFAA-torus} and it 
is a special case of \eqref{EQ:A-quant}, where we have identified, as noted,
$\widehat{\Tn}$ with $\Zn$.
\end{rem}
 
As a consequence of Theorem \ref{main125}  on the torus, we obtain:
\begin{cor}\label{cor125} Let $1\leq p_1,p_2 <\infty$. 
Let $A:L^{p_1}(\tn)\to L^{p_2}(\tn)$ be a linear
continuous operator with symbol $\sigma_A(\xi)$
depending only on $\xi.$ Then
 $A$ is $r$-nuclear provided that its symbol $\sigma_A$ satisfies 
\beq\label{EQ:cond-torus}
\sum_{\xi\in\zn} |\sigma_A(\xi)|^r<\infty.
\eq
\end{cor}



We shall now consider more general operators so that the symbols may depend also on $x$.

\begin{thm}\label{main12} 
Let $G$ be compact Lie group and $0<r\leq 1$. Let operator $A$ have the matrix symbol $\sigma_A(x,\xi)$. 
Let $1\leq p_1,p_2<\infty$ and let us denote $\widetilde{p_1}=\min\{2,p_1\}.$
Suppose that the symbol $\sigma_A$ satisfies 
$$\sum_{[\xi]\in\Gh}d_{\xi}^{2+\frac{r}{\widetilde{p_1}}}\|\|(\sigma_A(x,\xi))^t\|_{op(\ell^{\infty},\ell^{\infty})}\|_{L^{p_2}(G)}^r<\infty.$$
Then the extension $A:L^{p_1}(G)\rightarrow L^{p_2}(G)$ is $r$-nuclear. 
\end{thm}
Here we have denoted $\|\|\sigma_A(x,\xi)\|_{op(\ell^{\infty},\ell^{\infty})}\|_{L^{p_2}(G)}=
\p{\int_G \|\sigma_A(x,\xi)\|_{op(\ell^{\infty},\ell^{\infty})}^{p_2} \ dx}^{\frac{1}{p_2}}.$

\begin{proof} The kernel of the operator $A$ is given by 
\[k(x,y)=\sum_{[\xi]\in\widehat{G} }d_{\xi} \Tr(\xi(x)\sigma_A(x,\xi)\xi(y)^* ),\]
and we will show that it is well-defined and has the tensor product form of Theorem \ref{ch2} that 
is required for the nuclearity. As before, to abbreviate the notation, we will write
$\sigma(x,\xi)$ for $\sigma_A(x,\xi)$.
We begin by writing
\[ \Tr(\xi(x)\sigma(x,\xi)\xi(y)^* )=\sum\limits_{i,j=1}^{d_{\xi}}(\xi(x)\sigma(x,\xi))_{ij}\overline{\xi(y)}_{ij},\]
and we set 
\[g_{\xi,ij}(x)=d_{\xi}(\xi(x)\sigma(x,\xi))_{ij},\,h_{\xi,ij}(y)=(\xi(y)^*)_{ji}=\overline{\xi(y)}_{ij}.\]
A similar argument like in (\ref{toper}) shows that 
\[|(\xi(x)\sigma(x,\xi))_{ij}|\leq \|(\sigma(x,\xi))^t\|_{op(\ell^{\infty},\ell^{\infty})}.\]
Hence 
\begin{align*}
\|g_{\xi,ij}(x)\|_{L^{p_2}(G)}^r=&\|d_{\xi}(\xi(x)\sigma(x,\xi))_{ij}\|_{L^{p_2}(G)}^r\\
\leq & d_{\xi}^r\|\|(\sigma(x,\xi))^t\|_{op(\ell^{\infty},\ell^{\infty})}\|_{L^{p_2}(G)}^r.
\end{align*}
Let $q_1$ be such that
$\frac{1}{p_1}+\frac{1}{q_1}=1$.  
Now, if we denote $\widetilde{q_1}=\max\{2,q_1\}$, we have
$\frac{1}{\widetilde{p_1}}+\frac{1}{\widetilde{q_1}}=1$.
According to (\ref{ineqA}), we have
\[\|\overline{\xi}_{ij}\|_{L^{q_1}(G)}^r\leq d_{\xi}^{-\frac{r}{\widetilde{q_1}}}.\]
Therefore, 
\begin{align*}
\sum\limits_{\xi, ij}\|g_{\xi,ij}(\cdot)\|_{L^{p_2}(G)}^r\|h_{\xi,ij}(\cdot)\|_{L^{q_1}(G)}^r &\leq \sum\limits_{\xi}
d_{\xi}^rd_{\xi}^{-\frac{r}{\widetilde{q_1}}} d_{\xi}^{2}\|\|(\sigma(x,\xi))^t\|_{op(\ell^{\infty},\ell^{\infty})}\|_{L^{p_2}(G)}^r\\
&= \sum\limits_{\xi}d_{\xi}^{2 +\frac{r}{\widetilde{p_1}}}\|\|(\sigma(x,\xi))^t\|_{op(\ell^{\infty},\ell^{\infty})}\|_{L^{p_2}(G)}^r\\
&< \,\infty ,
\end{align*}
completing the proof.
\end{proof}
As a consequence of Theorem \ref{main12}, recalling the notation on the torus in 
Remark \ref{REM:torus}, for the torus group $G=\Tn$, we have:

\begin{cor}\label{main12ab} 
Let $1\leq p_1,p_2 <\infty$ and let $A:L^{p_1}(\Tn)\to L^{p_2}(\Tn)$ be a linear
continuous operator with symbol $\sigma_A(x,\xi)$ satisfying 
$$\sum_{\xi\in\Zn}\|\sigma_A(\cdot,\xi)\|_{L^{p_2}(\Tn)}^r<\infty,$$
then the operator $A:L^{p_1}(G)\rightarrow L^{p_2}(G)$ is $r$-nuclear. 
\end{cor}

\begin{rem}\label{REM:torus-nuclear}
 (i) If $G=\Tn=\Rn/ \Zn$, an invariant operator $A$ is a Fourier multiplier,
 $\widehat{Af}(k)=a(k)\widehat{f}(k)$ with symbol  
 $\sigma_A(\xi_k)=a(k)$, see Remark \ref{REM:torus}.
 Theorem  \ref{main12} implies that if
 $\sum_{k\in\Zn} |a(k)|^r<\infty$, then 
the operator $T$ is
 $r$-nuclear from $L^{p_1}(\Tn)$ to $L^{p_2}(\Tn)$ for all $1\leq p_1,p_2<\infty$.
 
 (ii) For the convolution operator on $\T$ as in \eqref{EQ:Carleman}, we have $\widehat{\T}\simeq {\mathbb Z}^1$ and
 $\sigma(n)=\widehat{\varkappa}(n)=c_n$, or $\varkappa(x)=\sum_{n=-\infty}^{\infty} c_n e^{2\pi i n x}.$
 In this case Theorem  \ref{main12} 
implies that if $\sum_{n=-\infty}^{\infty} |c_n|^r<\infty$, the operator $Tf=f*\varkappa$ is
 $r$-nuclear from $L^{p_1}(\T)$ to $L^{p_2}(\T)$ for all $1\leq p_1,p_2<\infty$.

(iii) If $p_1=p_2=2$ and $A$ is a left-invariant operator on a compact Lie group $G$, it follows
from Theorem  \ref{main125}
that if $\sum_{[\xi]\in\Gh} d_{\xi} \|\sigma_A(\xi)\|_{S_1}<\infty,$
then $A$ is a trace class operator on $L^{2}(G).$

(iv) We note that the condition of Corollary \ref{main12ab}  required the integrability of the symbol with
respect to $x$ and does not require any regularity.
 \end{rem}

 In order to deduce some interesting consequences we will apply the following lemma proved in \cite{dr:gevrey}:
\begin{lem}\label{lem2} 
Let $G$ be a compact Lie group. Then we have 
$$\sum_{[\xi]\in \widehat{G}}d_{\xi}^2\jp{\xi}^{-s}<\infty$$ 
if and only if $s>\dim G$. 
\end{lem}

This yields the following corollary, and in Remark \ref{REM:torus2} we note that the 
following orders are in general sharp.

\begin{cor}\label{main17} Let $G$ be a compact Lie group of dimension $n$ and let $0<r\leq 1$. 
Let $1\leq p_1,p_2<\infty$ and let us denote $\widetilde{p_1}=\min\{2,p_1\}.$
Assume that  
$$\|(\sigma_A(x,\xi))^t\|_{op(\ell^{\infty},\ell^{\infty})}\leq Cd_{\xi}^{-\frac{1}{\widetilde{p_1}}}\jp{\xi}^{-\frac{s}{r}}$$ 
with some $s>n$. Then $A:L^{p_1}(G)\rightarrow L^{p_2}(G)$ is $r$-nuclear.
\end{cor}
\begin{proof} We have 
\begin{align*} d_{\xi}^{2+\frac{r}{\widetilde{p_1}}}\|(\sigma(x,\xi))^t\|_{op(\ell^{\infty},\ell^{\infty})}^r \leq \,
&Cd_{\xi}^{2+\frac{r}{\widetilde{p_1}}}d_{\xi}^{-\frac{r}{\widetilde{p_1}}}\jp{\xi}^{-s}
=C^r d_{\xi}^2\jp{\xi}^{-s}.
\end{align*}
 The result now follows from Lemma \ref{lem2} and Theorem \ref{main12}. 
\end{proof}
As consequence of Theorem \ref{main126} and Lemma \ref{lem2} we have: 
\begin{cor}\label{main126a} 
Let $G$ be a compact Lie group, $1\leq p_1,p_2 <\infty$ and $0<r\leq 1$. Let $A:L^{p_1}(G)\to L^{p_2}(G)$ be a linear
continuous operator with matrix-valued diagonal symbol $\sigma_A(\xi)$ depending only on $\xi$.   
If   
$$\|\sigma_A(\xi)\|_{S_r} \leq Cd_{\xi}^{\frac 1r-(\frac{1}{\widetilde{p_1}}-\frac{1}{\widetilde{p_2}})}\jp{\xi}^{-\frac{s}{r}},$$
where $\widetilde{p_1}=\min\{2,p_1\}, \widetilde{p_2}=\max\{2,p_2\}$. Then the operator $A:L^{p_1}(G)\rightarrow L^{p_2}(G)$ is $r$-nuclear. 
\end{cor}
\begin{proof}  
\begin{align*} d_{\xi}^{1+(\frac{1}{\widetilde{p_1}}-\frac{1}{\widetilde{p_2}})r}\|\sigma_A(\xi)\|_{S_r}^r \leq \,
&C^rd_{\xi}^{1+(\frac{1}{\widetilde{p_1}}-\frac{1}{\widetilde{p_2}})r}d_{\xi}^{1-(\frac{1}{\widetilde{p_1}}-\frac{1}{\widetilde{p_2}})r}\jp{\xi}^{-s}
=C^r d_{\xi}^2\jp{\xi}^{-s}.
\end{align*}
 The result now follows from Lemma \ref{lem2} and Theorem \ref{main126}. 
\end{proof}
\subsection{Example on the torus}

We observe that on the torus $\tn$ criteria obtained in the above statements are in general sharp. 
In general, we recall that the relation of our setting to the special case of the torus was
outlined in Remark \ref{REM:torus}, with examples given already in Corollary \ref{cor125} 
and in Corollary \ref{main12ab}.

To see the sharpness, 
we establish the following simple characterisation of the nuclearity 
for Bessel potentials on $L^2(\tn)$.

\begin{prop}\label{PROP:torus}
Let $\Delta$ be the Laplacian on the torus $\Tn$ and let $0<r\leq 1$. Then
$(I-\Delta)^{-\frac{\alpha}{2}}$ is $r$-nuclear on $L^2(\tn)$ if and only if $\alpha r>n.$
\end{prop}
\begin{proof} 
The symbol of the operator $T=(I-\Delta)^{-\frac{\alpha}{2}}$ is positive, 
hence $T$ being a multiplier operator, it is positive definite and 
 $|T|=\sqrt{T^*T}=T$. Thus, the singular values of $T$ agree with the values of its symbol 
 $ \jp{\xi}^{-\alpha}$. Therefore, $T\in S_r(L^2(\Tn)$ if and only if $\alpha r>n.$ The result now follows from the identification
 of the Schatten class of order $r$ and the class of $r-$nuclear operators \cite{Oloff:pnorm}.   
\end{proof}

\begin{rem}\label{REM:torus2}
In the case of the torus $\Tn$ we have $d_{\xi}=1$. From Proposition \ref{PROP:torus}
it follows that the index $n$ in the sufficient condition in Corollary \ref{main17} cannot be improved. 
\end{rem}

\begin{cor}\label{main12abc} 
Let $1\leq p_1,p_2 <\infty$, $0<r\leq 1$ and let $A:L^{p_1}(\Tn)\to L^{p_2}(\Tn)$ be a linear
continuous operator with symbol $\sigma_A(x,\xi)$ satisfying 
$$\|\sigma_A(x,\xi)\|_{L^{p_2}(\Tn)}\leq C\langle\xi\rangle^{-s/r},$$
for some $s>n$. Then the operator $A:L^{p_1}(\Tn)\rightarrow L^{p_2}(\Tn)$ is $r$-nuclear $(\mbox{for all } p_1,p_2)$. 
\end{cor}
Using compactness of $G$, the following criterion can be practical:
\begin{cor}\label{main12abcd} 
Let $1\leq p_1,p_2 <\infty$, $0<r\leq 1$ and let $A:L^{p_1}(\Tn)\to L^{p_2}(\Tn)$ be a linear
continuous operator with symbol $\sigma_A(x,\xi)$ satisfying 
$$|\sigma_A(x,\xi)|\leq C\langle\xi\rangle^{-s/r},$$
for some $s>n$. Then the operator $A:L^{p_1}(\Tn)\rightarrow L^{p_2}(\Tn)$ is $r$-nuclear  $(\mbox{for all } p_1,p_2)$. 
\end{cor}

\subsection{Examples on $\SU2\simeq \mathbb S^{3}$ and on $\SO3$}
\label{SEC:SU2}

Let us now show other examples of the above statements for some particular compact groups. We first consider the case of $G=\SU2$, the group of the unitary $2\times 2$ matrices of determinant one. 
 The same results as given below can be stated for the 3-sphere $\mathbb S^{3}$ by using of    
the identification $\SU2\simeq\mathbb S^{3}$, with the matrix multiplication in 
$\SU2$ corresponding to the quaternionic product on $\mathbb S^{3}$,
with the corresponding identification of the symbolic calculus, see
\cite[Section 12.5]{rt:book}. 

 We refer the reader to \cite[Chapter 12]{rt:book} for the details of the global quantization
\eqref{EQ:A-quant} on $\SU2$ an the details on the representation theory of the group $G=\SU2$.
In this case, we can enumerate the elements of its dual as $\widehat{G}\simeq \frac12 \N_0$, 
with $\N_0=\{0\}\cup\mathbb N$, so that
$$
\widehat{\SU2}=\{ [t^\ell]: t^\ell\in\C^{(2\ell+1)\times (2\ell+1)},  \ell\in \frac12 \N_0\}.
$$
The dimension of each $t^\ell$ is $d_{t^\ell}=2\ell+1$, and there are explicit formulae for $t^\ell$ 
as functions of
Euler angles in terms of the so-called Legendre-Jacobi polynomials, see
\cite[Chapter 11]{rt:book}. The Laplacian on $\SU2$ has eigenvalues
$\lambda_{t^\ell}^2=\ell(\ell+1)$, so that we have $\jp{t^\ell}\approx \ell$.
With this, Corollary \ref{main17} becomes:

\begin{cor}\label{COR:SU2} 
Let $A:L^{p_1}(\SU2)\to L^{p_2}(\SU2)$ be an operator with matrix symbol 
$$
\sigma_A(x,\ell)\equiv \sigma_A(x,t^\ell):=t^\ell(x)^* At^\ell(x),\quad \ell\in\frac12\N_0.
$$
Let $s>3$ and $\widetilde{p_1}=\min\{2,p_1\}$.
If there is a constant $C>0$ such that 
$$
\|\|(\sigma_A(x,\ell))^t\|_{op(\ell^{\infty},\ell^{\infty})}\|_{L^{p_2}(G)}\leq C \ell^{-\frac{1}{\widetilde{p_1}}-\frac{s}{r}}
$$ 
for all $\ell\in\frac12\N$, then 
$A:L^{p_1}(\SU2)\rightarrow L^{p_2}(\SU2)$ is $r$-nuclear.
\end{cor}
For left-invariant operators with diagonal symbols on $\SU2$, as a consequence of 
 Corollary \ref{main126a} we have: 
\begin{cor}\label{COR3:SU2} 
Let $1\leq p_1,p_2<\infty$, $0<r\leq 1$ and let $\widetilde{p_1}=\min\{2,p_1\}$ and $\widetilde{p_2}=\max\{2,p_2\}$.
Let $A:L^{p_1}(\SU2)\to L^{p_2}(\SU2)$ be an operator with diagonal symbol $\sigma_A(\ell)$ such that
\[\|\sigma_A(\ell)\|_{S_r}\leq C\ell^{\frac{1-s}{r}-(\frac{1}{\widetilde{p_1}}-\frac{1}{\widetilde{p_2}})},\]
for some $s>3$. Then  
$A:L^{p_1}(\SU2)\rightarrow L^{p_2}(\SU2)$ is $r$-nuclear.
\end{cor}
In particular we will apply the above corollary to the Laplacian and the
sub-Laplacian.

If $\lapsu2$ denotes the Laplacian on $\SU2$, 
we have $\lapsu2 t^{\ell}_{mn}(x)=-\ell(\ell+1) t^{\ell}_{mn}(x)$ for all $\ell,m,n$ and $x\in G$,
so that the symbol of $I-\lapsu2$ is given by
$$\sigma_{I-\lapsu2}(x,\ell)=(1+\ell(\ell+1)) I_{2\ell+1},$$ 
where
$I_{2\ell+1}\in \C^{(2\ell+1)\times (2\ell+1)}$ is the identity matrix.
 Hence, $\sigma_{I-\lapsu2}(x,\ell)$ is diagonal and independent of $x$. 
 Consequently, Corollary \ref{COR3:SU2} applied to 
 $1\leq p=p_1=p_2<\infty$  says that the operator
$(I-\lapsu2)^{-\frac\alpha{2}}$ is $r$-nuclear on $L^p(\SU2)$ provided that 
$\ell^{\frac{1}{r}-\alpha}\leq C \ell^{\frac{1-s}r}$ for $s>3$. Summarising, we obtain

\begin{cor}\label{COR:su2-Lap}
For $\alpha > \frac 3r$, $0<r\leq 1$ and $1\leq p<\infty$, the operator
$(I-\lapsu2)^{-\frac{\alpha}{2}}$ is $r$-nuclear on $L^p(\SU2)$.
\end{cor}
%

  
If $p=2$, the order $\alpha r>3$ is sharp, see \cite{dr13:schatten}.  
 
\medskip 
We shall now consider the group $\SO3$ of the $3\times 3$ real orthogonal matrices 
of determinant one. For the details of the representation theory and the global 
quantization of $\SO3$ we refer the reader to \cite[Chapter 12]{rt:book}.
 The dual in this case can be identified as $\widehat{G}\simeq \ene_0$, so that
$$
\widehat{\SO3}=\{ [t^\ell]: t^\ell\in\C^{(2\ell+1)\times (2\ell+1)},  \ell\in \N_0\}.
$$
The dimension of each $t^\ell$ is $d_{t^\ell}=2\ell+1$. The Laplacian on $\SO3$ has eigenvalues
$\lambda_{t^\ell}^2=\ell(\ell+1)$, so that we have $\jp{t^\ell}\approx \ell$.
By the same argument as above, Corollary \ref{COR:su2-Lap}
also holds for the Laplacian on $\SO3$.

Let us fix three invariant vector fields $D_1, D_2, D_3$ on $\SO3$ 
corresponding to the derivatives with
respect to the Euler angles. We refer to \cite[Chapter 11]{rt:book} for the explicit formulae for these.
However, for our purposes here we note that the sub-Laplacian $\mathcal L_{sub}=D_1^2+D_2^2$,
with an appropriate choice of basis in the representation spaces, has the diagonal symbol given by
\begin{equation}\label{EQ:SO3-subLap}
\sigma_{\mathcal L_{sub}}(\ell)_{mn}=(m^2-\ell(\ell+1))\delta_{mn},\quad m,n\in\mathbb Z,\;
-\ell\leq m,n\leq\ell,
\end{equation}
where $\delta_{mn}$ is the Kronecker delta.
The operator $\mathcal L_{sub}$ is a second order hypoelliptic operator and we can define the
powers $(I-\mathcal L_{sub})^{-\alpha/2}$. These are pseudo-differential operators with symbols
$$\sigma_{(I-\mathcal L_{sub})^{-\alpha/2}}(\ell)_{mn}=(1+\ell(\ell+1)-m^2)^{-\alpha/2}\delta_{mn}.$$
We now have
$$
\|\sigma_{(I-\mathcal L_{sub})^{-\alpha/2}}(\ell)\|_{S_r}=
\left(\Tr (\sigma_{(I-\mathcal L_{sub})^{-\alpha/2}}(\ell))^r\right)^{\frac 1r}=
\left(\sum_{m=-\ell}^\ell\left(1+\ell(\ell+1) - m^2\right)^{-\frac{\alpha r}{2}}\right)^{\frac 1r},$$
where $\ell\in\N_0$. Comparing with the integral 
$$
\int_{-R}^R(1+R^2-x^2)^{-\frac{\alpha r}{2}}dx\approx C R^{-\frac{\alpha r}{2}}
\int_0^R(1+R-x)^{-\frac{\alpha r}{2}}dx\approx C R^{-\frac{\alpha r}{2}},
$$
for $\alpha r>2$ and large $R$,
it follows that $\sum_{m=-\ell}^\ell\left(1+\ell(\ell+1) - m^2\right)^{-\frac{\alpha r}{2}}$ 
is of order $\ell^{-\frac{\alpha r}{2}}$.  Now, the inequality
\[\ell^{-\frac{\alpha }{2}}\leq C\ell^{(1-s)/r},\] 
 with $s>3$ holds if and only $\alpha >\frac 4r$. 
 If $1\leq p=p_1=p_2<\infty$, as a consequence of 
Corollary \ref{COR3:SU2} we obtain the condition for the
nuclearity of the operator $(1-\mathcal L_{sub})^{-\alpha/2}$:

\begin{cor}\label{COR:su2-subLap}
For $\alpha> \frac 4r$ with $0<r\leq 1$ and $1\leq p<\infty$, the operator
$(I-\mathcal L_{sub})^{-\frac{\alpha}{2}}$ is $r$-nuclear on $L^p(\SO3)$.
The same conclusion holds for the same powers (of sub-Laplacians) 
on $L^p(\SU2)$ or on $L^p(\mathbb S^3)$.
\end{cor}
 
\section{Trace formulae on $L^p(G)$ and distribution of eigenvalues} 
\label{SEC:nuclearity-pr}

We now turn to some applications of the $r$-nuclearity on $L^p(G)$ spaces for the trace formulae, the Lidskii formula and the distribution of eigenvalues. In the special case $1\leq p_1=p_2=p<\infty$, applying Theorem \ref{THM:evs} and Theorem \ref{main12} we obtain:



\begin{cor}\label{main13ac} Let $G$ be compact Lie group and $0<r\leq 1$.
Let $1\leq p<\infty$ and let us denote $\widetilde{p}=\min\{2,p\}.$
Let $\sigma_A(x,\xi)$ be the matrix symbol of a bounded operator $A:L^p(G)\to L^p(G)$ such that
$$\sum_{[\xi]\in\Gh}d_{\xi}^{2+\frac{r}{\widetilde{p}}}\|\|(\sigma_A(x,\xi))^t\|_{op(\ell^{\infty},\ell^{\infty})}\|_{L^{p}(G)}^r<\infty.$$
Then $A:L^{p}(G)\rightarrow L^{p}(G)$ is $r$-nuclear and 
\[\sum\limits_{n=1}^{\infty}|\lambda_n(A)|^{\frac{2r}{2-r}}<\infty .\]
\end{cor}

We now derive another consequence relating the trace formulae with matrix-valued symbols. As we have already explained in the introduction, every nuclear operator defined from a Banach space $E$ into $E$ admits a trace provided that $E$ satisfies the approximation property, which is the case here dealing with $L^p$ spaces. In the next proposition we show that, when $p=p_1=p_2$ the sufficient condition in Theorem \ref{main12} ensures the existence of a formula for the trace in terms of the matrix-valued symbol. 
\begin{thm}\label{lem20} 
Let $G$ be a compact Lie group and $0<r\leq 1$. Let $1\leq p <\infty$ and $\widetilde{p}=\min\{2,p\}.$  
Let $A:L^{p}(G)\to L^{p}(G)$ be a linear
continuous operator with matrix-valued symbol $\sigma_A(x,\xi)$
such that
$$\sum_{[\xi]\in\Gh}d_{\xi}^{2 +\frac{r}{\widetilde{p}}}\|\|(\sigma_A(x,\xi))^t\|_{op(\ell^{\infty},\ell^{\infty})}\|_{L^{p}(G)}^r<\infty.$$
Then the operator $A:L^{p}(G)\rightarrow L^{p}(G)$ is $r$-nuclear and its trace is given by 
\beq 
\label{trace}\Tr A=\int_{G}\sum\limits_{[\xi]\in\widehat{G} }d_{\xi}\Tr(\sigma_A(x,\xi))d\mu(x) .
\eq
\end{thm}

\begin{proof}  The $r$-nuclearity is a consequence of Theorem \ref{main12}
and we adopt the notation of the proof of Theorem \ref{main12}, and denote $\sigma=\sigma_A$.
 Concerning the trace formula, in sake of simplicity we will just consider $r=1$ the general case follows from inclusion. As we have seen in the proof of Theorem \ref{main12}, the formula 
\[
k(x,y)=\sum\limits_{[\xi]\in\widehat{G} }d_{\xi} \Tr(\xi(x)\sigma(x,\xi)\xi(y)^* )
\]
represents the kernel of $A$. Moreover, it is well defined on the diagonal: in fact for the terms of 
the decomposition of the kernel
 \[g_{\xi,ij}(x)=d_{\xi}(\xi(x)\sigma(x,\xi))_{ij},\,h_{\xi,ij}(y)=(\xi(y)^*)_{ji}=\overline{\xi(y)}_{ij},\]
 by H\"older's inequality we have on the diagonal
\[\int_{G}|g_{\xi,ij}(x)| |h_{\xi,ij}(x)|d\mu (x)\leq  \|g_{\xi,ij}(\cdot)\|_{L^p(G)} \|h_{\xi,ij}(\cdot)\|_{L^q(G)}.\]
Hence, since $p=p_1=p_2$ we have
\begin{align*} 
\int_{G}\sum\limits_{[\xi]\in\widehat{G} }d_{\xi}|\Tr(\sigma(x,\xi))|d\mu(x)\leq & \sum\limits_{\xi, ij}\|g_{\xi,ij}(\cdot)\|_{L^{p}(G)}\|h_{\xi,ij}(\cdot)\|_{L^{q}(G)}\\
\leq & \sum\limits_{\xi}d_{\xi}^{2+\frac{1}{\widetilde{p}}}\|\|(\sigma(x,\xi))^t\|_{op(\ell^{\infty},\ell^{\infty})}\|_{L^{p}(G)}\\
<& \,\infty .
\end{align*}
Therefore,
\begin{align*} \Tr A=& \int_{G}k(x,x)d\mu(x)\\
=& \int_{G}\sum\limits_{[\xi]\in\widehat{G} }d_{\xi} \Tr(\xi(x)\sigma(x,\xi)\xi(x)^* )d\mu(x)\\
=&\int_{G}\sum\limits_{[\xi]\in\widehat{G} }d_{\xi} \Tr(\sigma(x,\xi)\xi(x)^*\xi(x) )d\mu(x)\\
=&\int_{G}\sum\limits_{[\xi]\in\widehat{G} }d_{\xi} \Tr(\sigma(x,\xi))d\mu(x).\\
\end{align*}
We have employed the tracial property $\Tr(AB)=\Tr(BA)$ and the fact that $\xi(x)$ is unitary for every $x$. 
\end{proof}
 In relation with the Lidskii formula, from Theorem \ref{lem20} and Grothendieck's theorem we have:
\begin{cor}\label{lem20al} 
Let $G$ be a compact Lie group and $0<r\leq \twothird$. Let $1\leq p <\infty$ and $\widetilde{p}=\min\{2,p\}.$  
Let $A:L^{p}(G)\to L^{p}(G)$ be a linear
continuous operator with matrix-valued symbol $\sigma_A(x,\xi)$
such that
$$\sum_{[\xi]\in\Gh}d_{\xi}^{2 +\frac{r}{\widetilde{p}}}\|\|(\sigma_A(x,\xi))^t\|_{op(\ell^{\infty},\ell^{\infty})}\|_{L^{p}(G)}^r<\infty.$$
Then the operator $A:L^{p}(G)\rightarrow L^{p}(G)$ is $r$-nuclear and its trace satisifies 
\beq 
\label{trace19}\Tr A=\int_{G}\sum\limits_{[\xi]\in\widehat{G} }d_{\xi}\Tr(\sigma_A(x,\xi))d\mu(x)=\sum\limits_{n=1}^{\infty}\lambda_n(A),
\eq
with multiplicities taken into account. 
\end{cor}

\begin{rem} We note that not for every kernel it
is convenient to calculate the trace integrating along the 
diagonal due to the degeneracy of the measure on it. When the kernel is representable 
by an expansion of the kind appearing in Theorem \ref{ch2} one is allowed 
to proceed in such a way. For a general kernel the integration along the diagonal 
should be calculated involving an averaging processes, see e.g. \cite{del:tracetop}.
\end{rem}

Very recently it has been proved (cf. \cite{Reinov}) that if $\frac{1}{r}=1+|\half -\frac{1}{p}|$, 
the Lidskii formula holds for  $r$-nuclear operators on $L^p(\nu)$-spaces. 
The importance of this result for us is that it allows to move $r$ along the interval $[\twothird, 1] $ 
keeping the validity of Lidskii's formula for  suitable values of $p$. If $r\in(\twothird, 1)$ 
there exist two corresponding values of $p$ solving the equation $\frac{1}{r}=1+|\half -\frac{1}{p}|$ 
the first one with $p<2$ and the other one with $p>2$. As a consequence of this result and 
 Theorem \ref{lem20} we obtain an extension of Corollary \ref{lem20al} allowing now a larger
 range of $r$:

 \begin{cor}\label{main13ab} 
 Let $G$ be compact Lie group. 
 Let $1\leq p<\infty$ and let us denote $\widetilde{p}=\min\{2,p\}.$
 Let $0<r\leq 1$ be such that $\frac{1}{r}=1+|\half -\frac{1}{p}|$. 
 If   
$$\sum_{[\xi]\in\Gh}d_{\xi}^{2+\frac{r}{\widetilde{p}}}\|\|(\sigma_A(x,\xi))^t\|_{op(\ell^{\infty},\ell^{\infty})}\|_{L^{p}(G)}^r<\infty,$$
 then $A$ is $r$-nuclear on $L^{p}(G)$ and we have
\[\Tr A=\int_{G}\sum\limits_{[\xi]\in\widehat{G} }d_{\xi}\Tr(\sigma_A(x,\xi))d\mu(x)=\sum\limits_{n=1}^{\infty}\lambda_n(A),
\]
with multiplicities taken into account. 
\end{cor} 

\subsection{Heat kernels}
\label{SEC:heat}
 
We shall now establish some applications, in particular to the heat kernels on compact Lie groups. 
The heat kernel constructions, and the subsequent Poisson kernel constructions, are instrumental
in the advances in the Littlewood-Paley theory on compact Lie groups, see e.g.
\cite{Stein:BOOK-topics-Littlewood-Paley}. However, our approach is more straightforward,
making use of the symbol of the heat kernel.
Indeed, taking into account that $\sigma_{e^{-t\lap}}(x,\xi)=e^{-t|\xi|^2}I_{d_{\xi}}$,  
where $|\xi|^2=\lambda_{[\xi]}^2$ with $\lambda_{[\xi]}$ as in \eqref{EQ:Lap-lambda},
we have
\begin{align*} 
e^{-t\lap}f(x)=&\sum\limits_{[\xi]\in \widehat{G}}d_{\xi} \Tr(\xi(x)\sigma_{e^{-t\lap}}(x,\xi)\widehat{f}(\xi))\\
=&\sum\limits_{[\xi]\in \widehat{G}}d_{\xi}e^{-t\eigen} \Tr(\xi(x)\widehat{f}(\xi)).
\end{align*}
We can now derive the nuclearity of the heat kernel on $L^p$-spaces.

\begin{thm}\label{main123} Let $G$ be compact Lie group. Then the heat operator
 $e^{-t\lap}:L^{p_1}(G)\rightarrow L^{p_2}(G)$ is nuclear for every $t>0$ and all
 $1\leq p_1,p_2<\infty$.
Moreover, if $0<r\leq 1$,
then 
 $e^{-t\lap}:L^{p}(G)\rightarrow L^{p}(G)$ is $r$-nuclear for every $t>0$ and $1\leq p<\infty$. 
 In particular, on each $L^p(G)$, due to the $1$-nuclearity we have the trace formula
 $$
 \Tr e^{-t\lap}=\sum_{[\xi]\in\Gh} d_\xi^2 e^{-t\lambda_{[\xi]}^2}.
 $$

\end{thm}
\begin{proof} The kernel of $e^{-t\lap}$ is given by 
\[k_t(x,y)=\sum\limits_{[\xi]\in\widehat{G} }d_{\xi}e^{-t\eigen} \Tr(\xi(x)\xi(y)^* ),\]
with
\[ \Tr(\xi(x)\xi(y)^* )=\sum\limits_{i,j=1}^{d_{\xi}} \xi(x)_{ij}\overline{\xi(y)}_{ij}.\]
We set 
\[g_{\xi,ij}(x)=d_{\xi}e^{-t\eigen}\xi(x)_{ij},\,h_{\xi,ij}(y)=(\xi(y)^*)_{ji}=\overline{\xi(y)}_{ij}.\]
As before we shall consider $q_1$ such that $\frac{1}{p_1}+\frac{1}{q_1}=1$ and we 
denote $\widetilde{q_1}=\max\{2,q_1\}$. Then by Lemma \ref{LEM:xiests} we have
\[\|\overline{\xi}_{ij}\|_{L^{q_1}(G)}= d_{\xi}^{-\frac{1}{\widetilde{q_1}}}.\]
 On the other hand 
\begin{align*}
\|g_{\xi,ij}\|_{L^{p_2}(G)}=&\|d_{\xi}e^{-t\eigen}\xi_{ij}\|_{L^{p_2}(G)}\\
\leq & \|d_{\xi}e^{-t\eigen}\|\xi\|_{op}\|_{L^{p_2}(G)}\\
\leq & {d_{\xi}} e^{-t\eigen}.
\end{align*}
Therefore,
$$
\sum\limits_{\xi, ij}\|g_{\xi,ij}(\cdot)\|_{L^{p_2}(G)}\|h_{\xi,ij}(\cdot)\|_{L^{q_1}(G)} 
\leq \sum\limits_{\xi} d_{\xi}^2{d_{\xi}}^{\frac{1}{\widetilde{p_1}}}e^{-t\eigen} 
< \,\infty ,
$$
the last convergence following, for example, from any of the Weyl formulae, see,
for example \cite{dr:gevrey}.

The $r$-nuclearity follows in a similar way.
The trace formula follows immediately from Lemma \ref{lem20} and fact that the Haar measure on $G$
is normalised:
$$
\Tr e^{-t\lap}=\int_G \sum_{[\xi]\in\Gh} d_\xi \Tr( e^{-t\lambda_{[\xi]}^2} I_{d_\xi})=
\sum_{[\xi]\in\Gh} d_\xi^2 e^{-t\lambda_{[\xi]}^2}.
$$
\end{proof}

\begin{rem} The Lidskii formula can be used to deduce lower bounds on the number of eigenvalues:
Let $E$ be a Banach space enjoying the approximation property. 
If $T:E\rightarrow E$ is a $\twothird$-nuclear operator which possesses at least one eigenvalue and 
if $|\lambda_k(T)|\leq M$ for all $k$, then 
\[\frac{|\Tr(T)|}{M}\leq N ,  \]
 where $N$ is the number of eigenvalues of $T$. Indeed, applying the Lidskii formula
   \[\Tr(T)=\sum\limits_{k=1}^{N}\lambda_k(T),\]
  we can estimate
     \[|\Tr(T)|= |\sum\limits_{k=1}^{N}\lambda_k(T)|\leq \sum\limits_{k=1}^{N}|\lambda_k(T)|\leq MN.\]  
As a consequence of this observation, taking into account the trace formula in \cite{del:tracetop} 
we obtain the following estimate for integral operates. The symbol $\,\widetilde{}\, $ will denote the 
averaging process for kernels described in \cite{del:tracetop}. 
Let $\mu$ be a Borel measure on a second countable topological space and let
$T:L^p(\mu)\rightarrow L^p(\mu)$ be a 
 $\twothird$-nuclear operator with kernel $K(x,y)$. If $T$ possesses at least one eigenvalue and 
 if $|\lambda_k(T)|\leq M$ for all $k$, then
 \[\frac{|\int\limits_{\Omega}{\tilde{K}}(x,x)d\mu(x)|}{M}\leq N ,  \]
 where $N$ is the number of eigenvalues of $T$.
  The last inequality means that the better one can estimate the size of the trace the better lower bound 
  one gets for the
 number of the eigenvalues.
\end{rem}  
 



\end{document}